\documentclass{amsart}
\usepackage{amssymb,amsfonts}
\usepackage{graphicx}
\usepackage{longtable,makecell,url}
\usepackage[hidelinks]{hyperref}
\newtheorem{theorem}{Theorem}
\newtheorem{lemma}[theorem]{Lemma}

\newtheorem{proposition}[theorem]{Proposition}
\newtheorem{remark}[theorem]{Remark}

\begin{document}

\title[Powers of the Thue--Morse Series]{2-adic Valuations of Coefficients
of the Fifth and Ninth Powers of the Thue--Morse Generating Function}

\author{Zhao Shen}
\author{Xinping Wang}

\address{Department of Mathematics, Central South University, Changsha, Hunan, China}
\email{sz1021@csu.edu.cn}

\date{June 26, 2026}

\begin{abstract}
Let $T(x)=\prod_{k=0}^{\infty}(1-x^{2^k})$ be the generating function of the Thue--Morse
sequence, and write $T(x)^m=\sum_{n\geq 0}t_m(n)x^n$.
We prove exact formulas for the $2$-adic valuations of the coefficients $t_5(n)$ and $t_9(n)$:
\[
\nu_2\bigl(t_5(4n+j)\bigr)
  =4\Bigl\lceil\tfrac{\nu_2(n+1)}{2}\Bigr\rceil-\bigl(\nu_2(n+1)\bmod 2\bigr),
\quad j\in\{0,1,2,3\},
\]
\[
\nu_2\bigl(t_9(8n+j)\bigr)
  =5\Bigl\lceil\tfrac{\nu_2(n+1)}{2}\Bigr\rceil-2\bigl(\nu_2(n+1)\bmod 2\bigr),
\quad j\in\{0,1,\ldots,7\}.
\]
These formulas confirm Conjecture~5.2 of Gawron--Miska--Ulas~\cite{ga} for $m=5$ and $m=9$,
and imply that $t_5(n)\neq 0$ and $t_9(n)\neq 0$ for every $n\geq 0$.
A key structural ingredient is a closed-form formula for the determinant of a
family of matrices with binomial-coefficient entries.
\end{abstract}

\maketitle

\section{Introduction}

The Thue--Morse sequence $\{t(n)\}_{n\geq 0}$ is defined by $t(0)=1$ and the
recurrences $t(2n)=t(n)$, $t(2n+1)=-t(n)$.
Its generating function is the infinite product
\begin{equation}\label{eq:T}
T(x)=\prod_{k=0}^{\infty}(1-x^{2^k}).
\end{equation}
For an integer $m\geq 1$ we write
\begin{equation}\label{eq:tm}
T(x)^m=\sum_{n=0}^{\infty}t_m(n)\,x^n,
\end{equation}
so $t_1(n)=t(n)$.  Arithmetic properties of $\{t_m(n)\}_{n\geq 0}$ were studied
systematically by Gawron, Miska and Ulas~\cite{ga}.  The unboundedness of the coefficients of
$T(x)^m$ is now known in full generality by the preprint~\cite{shen-unbounded}.  By contrast,
the general case of Conjecture~5.2 of~\cite{ga}, which concerns exact $2$-adic valuation formulas
for odd exponents, remains open.

The case $m=5$ was worked out jointly by Xinping Wang and Zhao Shen and first
appeared in the undergraduate thesis~\cite{wang2026thesis}.
The case $m=9$ requires a more delicate congruence analysis.
Both proofs rely on block matrix recurrences derived from the functional equation
$T(x)=(1-x)T(x^2)$, together with the following determinant identity.

\begin{theorem}\label{thm:det}
For any integer $m\ge 2$, let $M_m$ be the $(m-1)\times(m-1)$ matrix whose $(i,j)$-entry is
\[
(M_m)_{i,j}=\binom{m}{2i-j},
\]
with the convention $\binom{m}{k}=0$ whenever $k<0$ or $k>m$.  Then
\[
\det(M_m)=2^{\,m(m-1)/2}.
\]
\end{theorem}

\begin{theorem}\label{thm:m5}
For every integer $n\ge 0$ and every $j\in\{0,1,2,3\}$,
\[
\nu_2\bigl(t_5(4n+j)\bigr)
  =4\Bigl\lceil\tfrac{\nu_2(n+1)}{2}\Bigr\rceil-\bigl(\nu_2(n+1)\bmod 2\bigr).
\]
In particular, $t_5(n)\neq 0$ for all $n\ge 0$.
\end{theorem}

\begin{theorem}\label{thm:m9}
For every integer $n\ge 0$ and every $j\in\{0,1,\ldots,7\}$,
\[
\nu_2\bigl(t_9(8n+j)\bigr)
  =5\Bigl\lceil\tfrac{\nu_2(n+1)}{2}\Bigr\rceil-2\bigl(\nu_2(n+1)\bmod 2\bigr).
\]
In particular, $t_9(n)\neq 0$ for all $n\ge 0$.
\end{theorem}

\noindent\textit{Organisation.}
Section~\ref{sec:prelim} recalls the block recurrences from~\cite{ga}.
Section~\ref{sec:det} proves Theorem~\ref{thm:det}.
Sections~\ref{sec:m5} and~\ref{sec:m9} prove
Theorems~\ref{thm:m5} and~\ref{thm:m9}, respectively.

\section{Preliminaries}\label{sec:prelim}

For integers $a,b\ge 0$, raising
\[
T(x)=\prod_{n=0}^{a}(1-x^{2^n})\cdot T(x^{2^{a+1}})
\]
to the $b$-th power gives
\begin{equation}\label{eq:h}
\sum_{n=0}^{\infty}t_b(n)\,x^n
  =\prod_{n=0}^{a}(1-x^{2^n})^b\cdot\sum_{n=0}^{\infty}t_b(n)\,x^{2^{a+1}n}.
\end{equation}
Comparing coefficients in~\eqref{eq:h} with $a=0$ yields (see~\cite[Lemma~3.1]{ga}):

\begin{lemma}[{\cite[Lemma 3.1]{ga}}]\label{lem:rec}
For $m\ge 1$ and $n\ge 0$ (with $t_m(n)=0$ for $n<0$),
\begin{align}
t_m(2n)   &=\sum_{j=0}^{\lfloor m/2\rfloor}\tbinom{m}{2j}\,t_m(n-j),\label{eq:even}\\[2pt]
t_m(2n+1) &=-\sum_{j=0}^{\lfloor(m-1)/2\rfloor}\tbinom{m}{2j+1}\,t_m(n-j).\label{eq:odd}
\end{align}
\end{lemma}

For every integer $m\ge 2$, let
\[
V^{(m)}_n:=\bigl(t_m(n-1),t_m(n-2),\ldots,t_m(n-m+1)\bigr)^T.
\]
Then Lemma~\ref{lem:rec} can be written in matrix form as
\begin{equation}\label{eq:matrix-rec}
V^{(m)}_{2n}=A_mV^{(m)}_n,
\end{equation}
where $A_m=(a_{ij})_{1\le i,j\le m-1}$ is the $(m-1)\times(m-1)$ matrix defined by
\begin{equation}\label{eq:Am-def}
a_{ij}=(-1)^i\binom{m}{2j-i},
\end{equation}
with the convention $\binom{m}{k}=0$ for $k<0$ or $k>m$.
Indeed, the $i$-th entry of $V^{(m)}_{2n}$ is $t_m(2n-i)$; using \eqref{eq:even} when $i$ is even
and \eqref{eq:odd} when $i$ is odd gives the same coefficient formula \eqref{eq:Am-def}.

\begin{proposition}\label{prop:block}
Let $m=2d+1$ be odd, and let $J$ be the reversal permutation matrix of size $m-1$.
Define vectors
\[
u_j:=e_j+e_{m-j},\qquad v_j:=e_j-e_{m-j},\qquad 1\le j\le d.
\]
Then, in the ordered basis
\[
(u_1,\ldots,u_d,\,v_1,\ldots,v_d),
\]
the matrix $A_m$ has the block form
\[
A_m\sim
\begin{pmatrix}
0 & C_m\\
D_m & 0
\end{pmatrix},
\]
where $C_m=(c^{+}_{ij})_{1\le i,j\le d}$ and $D_m=(c^{-}_{ij})_{1\le i,j\le d}$ are given by
\[
c^{+}_{ij}=(-1)^i\left(\binom{m}{2j-i}-\binom{m}{2j+i-m}\right),
\]
\[
c^{-}_{ij}=(-1)^i\left(\binom{m}{2j-i}+\binom{m}{2j+i-m}\right).
\]
Consequently,
\[
A_m^2\sim
\begin{pmatrix}
C_mD_m & 0\\
0 & D_mC_m
\end{pmatrix},
\]
and hence
\[
\chi_{A_m^2}(Y)=\chi_{C_mD_m}(Y)^2=\chi_{D_mC_m}(Y)^2.
\]
\end{proposition}

\begin{proof}
Since $JA_mJ=-A_m$, the matrix $A_m$ interchanges the $\pm 1$ eigenspaces of $J$.
These eigenspaces are spanned by $(u_1,\ldots,u_d)$ and $(v_1,\ldots,v_d)$, respectively.
Thus $A_m$ has the form
\[
\begin{pmatrix}
0 & C_m\\
D_m & 0
\end{pmatrix}
\]
in that basis.

To compute $C_m$ and $D_m$, note that the coefficient of $v_i$ in $A_mu_j$ is obtained
from the $i$-th coordinate of $A_m(e_j+e_{m-j})$, namely
\[
(A_m)_{i,j}+(A_m)_{i,m-j},
\]
while the coefficient of $u_i$ in $A_mv_j$ is
\[
(A_m)_{i,j}-(A_m)_{i,m-j}.
\]
Using
\[
(A_m)_{i,j}=(-1)^i\binom{m}{2j-i}
\]
and
\[
(A_m)_{i,m-j}=(-1)^i\binom{m}{2(m-j)-i}=(-1)^i\binom{m}{2j+i-m},
\]
we obtain the displayed formulas for $C_m$ and $D_m$.

Squaring the block matrix gives
\[
\begin{pmatrix}
0 & C_m\\
D_m & 0
\end{pmatrix}^2=
\begin{pmatrix}
C_mD_m & 0\\
0 & D_mC_m
\end{pmatrix}.
\]
Finally, since $C_m$ and $D_m$ are square matrices of the same size, the standard identity
\[
\chi_{C_mD_m}(Y)=\chi_{D_mC_m}(Y)
\]
implies
\[
\chi_{A_m^2}(Y)=\chi_{C_mD_m}(Y)\chi_{D_mC_m}(Y)=\chi_{C_mD_m}(Y)^2.
\qedhere
\]
\end{proof}

\begin{remark}
Let $J$ be the reversal permutation matrix of size $m-1$, i.e.
$J_{i,j}=1$ if $i+j=m$ and $J_{i,j}=0$ otherwise.  From
\[
a_{ij}=(-1)^i\binom{m}{2j-i}
\]
one checks directly that, for odd $m$,
\[
JA_mJ=-A_m.
\]
Hence $A_m$ is similar to $-A_m$, so its characteristic polynomial is even:
\[
\chi_{A_m}(X)=\chi_{A_m}(-X).
\]
In particular, the occurrence of only even powers in the characteristic and minimal
polynomials of $A_m$ is forced by this symmetry.
\end{remark}

\section{A determinant formula}\label{sec:det}

\begin{proof}[Proof of Theorem~\ref{thm:det}]
We proceed by induction on $m$.

\paragraph{Base case.}
For $m=2$, $M_2=\bigl[\binom{2}{1}\bigr]=[2]$, so $\det(M_2)=2=2^{2\cdot 1/2}$.

\paragraph{Step 1: Matrix factorisation.}
By Pascal's identity $\binom{m}{k}=\binom{m-1}{k}+\binom{m-1}{k-1}$, write
$M_m=B_mC_m$,
where $B_m$ is the $(m-1)\times m$ matrix with entries
$(B_m)_{i,j}=\binom{m-1}{2i-j}$ ($1\le i\le m-1$, $1\le j\le m$),
and $C_m$ is the $m\times(m-1)$ matrix with $(C_m)_{k,j}=1$ if $k\in\{j,j+1\}$
and $0$ otherwise. Indeed,
\[
(B_mC_m)_{ij}=\binom{m-1}{2i-j}+\binom{m-1}{2i-j-1}=\binom{m}{2i-j}=(M_m)_{ij}.
\]

\paragraph{Step 2: Cauchy--Binet.}
Deleting the $k$-th row of $C_m$ leaves a unit upper-triangular $(m-1)\times(m-1)$ matrix, so
$\det\bigl((C_m)^{\hat k}\bigr)=1$ for every $k$.
Cauchy--Binet therefore gives
\begin{equation}\label{eq:cb}
\det(M_m)=\sum_{k=1}^{m}\det\bigl((B_m)_{\hat k}\bigr).
\end{equation}

\paragraph{Step 3: Kernel of $B_m$.}
Define $\mathbf{v}=(v_1,\dots,v_m)^T$ by $v_j=(-1)^{j-1}\binom{m-1}{j-1}$.
The $i$-th component of $B_m\mathbf{v}$ equals the coefficient of $x^{2i-1}$ in
$(1+x)^{m-1}(1-x)^{m-1}=(1-x^2)^{m-1}$,
which vanishes because $2i-1$ is odd.
Hence $\mathbf{v}\in\ker B_m$; since $\operatorname{rank}B_m=m-1$ the kernel is one-dimensional, so
the vector of signed maximal minors
\[
\bigl((-1)^{k-1}\det\bigl((B_m)_{\hat k}\bigr)\bigr)_{1\le k\le m}
\]
also spans $\ker B_m$.  Therefore there exists a constant $c$ such that
\[
(-1)^{k-1}\det\bigl((B_m)_{\hat k}\bigr)=c(-1)^{k-1}\binom{m-1}{k-1},\qquad k=1,\dots,m,
\]
and hence
\[
\det\bigl((B_m)_{\hat k}\bigr)=c\binom{m-1}{k-1},\qquad k=1,\dots,m.
\]

\paragraph{Step 4: Determine $c$.}
Taking $k=m$: the submatrix $(B_m)_{\hat m}$ has entries $\binom{m-1}{2i-j}$ with
$1\le i,j\le m-1$, which is exactly $M_{m-1}$. Hence $c=\det(M_{m-1})$.

\paragraph{Step 5: Recurrence.}
Substituting into~\eqref{eq:cb},
\[
\det(M_m)=\det(M_{m-1})\sum_{k=1}^{m}\binom{m-1}{k-1}=\det(M_{m-1})\cdot 2^{m-1}.
\]
Iterating from the base case:
$\det(M_m)=2^{m-1}\cdot 2^{m-2}\cdots 2^1=2^{m(m-1)/2}.$
\end{proof}

\section{The case \texorpdfstring{$m=5$}{m=5}}\label{sec:m5}

By \eqref{eq:Am-def}, the matrix $A_5$ from Lemma~\ref{lem:rec} is
\[
A_5=
\begin{pmatrix}
-5 & -10 & -1 & 0\\
1 & 10 & 5 & 0\\
0 & -5 & -10 & -1\\
0 & 1 & 10 & 5
\end{pmatrix}.
\]
Let
\[
U_n:=\bigl(t_5(4n-1),t_5(4n-2),t_5(4n-3),t_5(4n-4)\bigr)^T.
\]
Applying $A_5^2$ to $V_n^{(5)}$ gives
\begin{equation}\label{eq:A5-square}
A_5^2=
\begin{pmatrix}
15 & -45 & -35 & 1\\
5 & 65 & -1 & -5\\
-5 & -1 & 65 & 5\\
1 & -35 & -45 & 15
\end{pmatrix},
\end{equation}
and another computation gives
\begin{equation}\label{eq:A5-fourth}
A_5^4=16
\begin{pmatrix}
11 & -225 & -175 & 5\\
25 & 261 & -5 & -25\\
-25 & -5 & 261 & 25\\
5 & -175 & -225 & 11
\end{pmatrix}
=80A_5^2-1024I_4.
\end{equation}
Applying \eqref{eq:A5-fourth} directly to $V_n^{(5)}$ yields
\begin{equation}\label{eq:t5scalar-alt}
t_5(16n-k)=80t_5(4n-k)-1024t_5(n-k),\qquad k\in\{1,2,3,4\}.
\end{equation}

Next, set $a=2$ and $b=5$ in~\eqref{eq:h}.  For later use we record the explicit expansion
{\footnotesize
\begin{equation}\label{eq:h25}
\setlength{\jot}{2pt}
\begin{aligned}
h^{(2,5)}(x):={}&\prod_{n=0}^{2}(1-x^{2^n})^5 \\
={}&(1-x)^5(1-x^2)^5(1-x^4)^5 \\
={}&1-5x+5x^2+15x^3-40x^4+24x^5+40x^6-120x^7+140x^8+20x^9\\
&-276x^{10}+340x^{11}-120x^{12}-280x^{13}+600x^{14}-424x^{15}-170x^{16}\\
&+610x^{17}-610x^{18}+170x^{19}+424x^{20}-600x^{21}+280x^{22}-120x^{23}\\
&-340x^{24}+276x^{25}-20x^{26}-140x^{27}+120x^{28}-40x^{29}-24x^{30}+40x^{31}\\
&-15x^{32}-5x^{33}+5x^{34}-x^{35}.
\end{aligned}
\end{equation}
}
Substituting \eqref{eq:h25} into~\eqref{eq:h} and comparing coefficients gives the parity relation
\begin{equation}\label{eq:t5-mod2}
t_5(8n+j)\equiv t_5(n)-t_5(n-4)\pmod{2},
\qquad j\in\{0,1,2,3\},
\end{equation}
and the identities
\begin{equation}\label{eq:t5-8n}
\begin{aligned}
t_5(8n-4) &= 8\bigl(-5t_5(n-1)-15t_5(n-2)+53t_5(n-3)+15t_5(n-4)\bigr),\\
t_5(8n-3) &= 8\bigl(3t_5(n-1)-35t_5(n-2)-75t_5(n-3)-5t_5(n-4)\bigr),\\
t_5(8n-2) &= 8\bigl(5t_5(n-1)+75t_5(n-2)+35t_5(n-3)-3t_5(n-4)\bigr),\\
t_5(8n-1) &= 8\bigl(-15t_5(n-1)-53t_5(n-2)+15t_5(n-3)+5t_5(n-4)\bigr).
\end{aligned}
\end{equation}

\begin{proof}[Proof of Theorem~\ref{thm:m5}]
Let
\[
g_5(r):=4\Bigl\lceil\frac{r}{2}\Bigr\rceil-(r\bmod 2).
\]
We prove by induction on $n\ge 0$ that
\[
\nu_2\bigl(t_5(4n+j)\bigr)=g_5\bigl(\nu_2(n+1)\bigr),
\qquad j\in\{0,1,2,3\}.
\]
The explicit expansion \eqref{eq:h25} gives $t_5(0),t_5(1),\ldots,t_5(7)$ immediately.
We compute $t_5(8),\ldots,t_5(15)$ directly as well.  Thus the theorem holds for the initial
block indices $0\le n\le 3$:
\[
\begin{aligned}
&(t_5(0),t_5(1),t_5(2),t_5(3))=(1,-5,5,15),\\
&(\nu_2(t_5(0)),\nu_2(t_5(1)),\nu_2(t_5(2)),\nu_2(t_5(3)))=(0,0,0,0),\\
&(t_5(4),t_5(5),t_5(6),t_5(7))=(-40,24,40,-120),\\
&(\nu_2(t_5(4)),\nu_2(t_5(5)),\nu_2(t_5(6)),\nu_2(t_5(7)))=(3,3,3,3),\\
&(t_5(8),t_5(9),t_5(10),t_5(11))=(135,45,-301,265),\\
&(\nu_2(t_5(8)),\nu_2(t_5(9)),\nu_2(t_5(10)),\nu_2(t_5(11)))=(0,0,0,0),\\
&(t_5(12),t_5(13),t_5(14),t_5(15))=(80,-400,400,176),\\
&(\nu_2(t_5(12)),\nu_2(t_5(13)),\nu_2(t_5(14)),\nu_2(t_5(15)))=(4,4,4,4).
\end{aligned}
\]
These agree with
\[
g_5\bigl(\nu_2(1)\bigr)=0,\quad g_5\bigl(\nu_2(2)\bigr)=3,\quad g_5\bigl(\nu_2(3)\bigr)=0,\quad g_5\bigl(\nu_2(4)\bigr)=4
\]
for $n=0,1,2,3$, respectively.

Assume the formula is true for all indices $<n$.  Write
\[
n=2^r n_1-1,\qquad n_1\equiv 1\pmod{2}.
\]

\medskip
\noindent\textbf{Case $r=0$.}
Then $n$ is even, say $n=2u$.  By \eqref{eq:t5-mod2},
\[
t_5(8u+j)\equiv t_5(u)-t_5(u-4)\pmod 2,
\qquad j\in\{0,1,2,3\}.
\]
Write $u=4q+s$ with $s\in\{0,1,2,3\}$.  Then
\[
t_5(u)=t_5(4q+s),\qquad t_5(u-4)=t_5(4(q-1)+s).
\]
By the induction hypothesis,
\[
\nu_2\bigl(t_5(u)\bigr)=g_5\bigl(\nu_2(q+1)\bigr),
\qquad
\nu_2\bigl(t_5(u-4)\bigr)=g_5\bigl(\nu_2(q)\bigr).
\]
Since one of $q$ and $q+1$ is odd and the other is even, exactly one of
$t_5(u)$ and $t_5(u-4)$ is odd. Hence $t_5(8u+j)$ is odd for every $j$, so
\[
\nu_2\bigl(t_5(4n+j)\bigr)=\nu_2\bigl(t_5(8u+j)\bigr)=0=g_5(0).
\]

\medskip
\noindent\textbf{Case $r=1$.}
Then $n=2n_1-1$ with $n_1$ odd.  Put
\[
S(x,y,z,w):=xt_5(n_1-1)+yt_5(n_1-2)+zt_5(n_1-3)+wt_5(n_1-4).
\]
We claim that $S(x,y,z,w)$ is odd whenever $x,y,z,w$ are odd.

If $n_1\equiv 1\pmod 4$, write $n_1=4q+1$.  Then
\[
\nu_2\bigl(t_5(n_1-1)\bigr)=g_5\bigl(\nu_2(q+1)\bigr),
\]
while
\[
\nu_2\bigl(t_5(n_1-2)\bigr)=\nu_2\bigl(t_5(n_1-3)\bigr)=\nu_2\bigl(t_5(n_1-4)\bigr)=g_5\bigl(\nu_2(q)\bigr).
\]
Again, one of $q$ and $q+1$ is odd, so exactly one of these four values is odd.

If $n_1\equiv 3\pmod 4$, write $n_1=4q+3$.  Then
\[
\nu_2\bigl(t_5(n_1-1)\bigr)=\nu_2\bigl(t_5(n_1-2)\bigr)=\nu_2\bigl(t_5(n_1-3)\bigr)=g_5\bigl(\nu_2(q+1)\bigr),
\]
while
\[
\nu_2\bigl(t_5(n_1-4)\bigr)=g_5\bigl(\nu_2(q)\bigr).
\]
So here as well exactly one of $t_5(n_1-1),\ldots,t_5(n_1-4)$ is odd.

This proves the claim.  Applying \eqref{eq:t5-8n} with $n=n_1$, each
parenthesised linear combination is odd because all coefficients are odd.  Therefore
\[
\nu_2\bigl(t_5(4n+j)\bigr)=3=g_5(1),
\qquad j\in\{0,1,2,3\}.
\]

\medskip
\noindent\textbf{Case $r=2$ or $r=3$.}
For $j\in\{0,1,2,3\}$, set $k:=4-j$.  Since
\[
4n+j=2^{r+2}n_1-k,
\]
equation \eqref{eq:t5scalar-alt} with $n=2^{r-2}n_1$ gives
\[
t_5(4n+j)=80t_5\bigl(4(2^{r-2}n_1-1)+j\bigr)-1024t_5\bigl(2^{r-2}n_1-4+j\bigr).
\]
By the induction hypothesis,
\[
\nu_2\bigl(t_5(4(2^{r-2}n_1-1)+j)\bigr)=g_5(r-2)<6.
\]
Hence the first term has valuation
\[
\nu_2\Bigl(80t_5\bigl(4(2^{r-2}n_1-1)+j\bigr)\Bigr)=4+g_5(r-2)=g_5(r),
\]
while the second term has valuation at least
\[
\nu_2\Bigl(1024t_5\bigl(2^{r-2}n_1-4+j\bigr)\Bigr)\ge 10>g_5(r).
\]
Therefore the first term dominates and
\[
\nu_2\bigl(t_5(4n+j)\bigr)=4+g_5(r-2)=g_5(r).
\]

\medskip
\noindent\textbf{Case $r\ge 4$.}
Using the same identity as in the previous case, the induction hypothesis gives
\[
\nu_2\bigl(t_5(4(2^{r-2}n_1-1)+j)\bigr)=g_5(r-2),
\qquad
\nu_2\bigl(t_5(2^{r-2}n_1-4+j)\bigr)=g_5(r-4).
\]
Thus the first term has valuation $4+g_5(r-2)=g_5(r)$, while the second term has valuation
\[
10+g_5(r-4)=g_5(r)+2>g_5(r).
\]
Hence
\[
\nu_2\Bigl(80t_5\bigl(4(2^{r-2}n_1-1)+j\bigr)\Bigr)=g_5(r)
<g_5(r)+2=
\nu_2\Bigl(1024t_5\bigl(2^{r-2}n_1-4+j)\bigr)\Bigr),
\]
so the first term dominates. Therefore
\[
\nu_2\bigl(t_5(4n+j)\bigr)=g_5(r)=g_5\bigl(\nu_2(n+1)\bigr).
\]
This completes the induction.
\end{proof}

\section{The case \texorpdfstring{$m=9$}{m=9}}\label{sec:m9}

Define the $8$-dimensional column vectors
\[
X_n:=\bigl(t_9(2n-1),\,t_9(2n-2),\,\ldots,\,t_9(2n-8)\bigr)^T,\quad n\ge 1,
\]
\[
U_n:=\bigl(t_9(8n-8),\,t_9(8n-7),\,\ldots,\,t_9(8n-1)\bigr)^T,\quad n\ge 1.
\]
Since $X_n=V^{(9)}_{2n}$, we have
\[
X_{2n}=A_9 X_n,
\]
where
{\scriptsize
\[
A_9=
\begin{pmatrix}
-9 & -84 & -126 & -36 & -1 & 0 & 0 & 0\\
1 & 36 & 126 & 84 & 9 & 0 & 0 & 0\\
0 & -9 & -84 & -126 & -36 & -1 & 0 & 0\\
0 & 1 & 36 & 126 & 84 & 9 & 0 & 0\\
0 & 0 & -9 & -84 & -126 & -36 & -1 & 0\\
0 & 0 & 1 & 36 & 126 & 84 & 9 & 0\\
0 & 0 & 0 & -9 & -84 & -126 & -36 & -1\\
0 & 0 & 0 & 1 & 36 & 126 & 84 & 9
\end{pmatrix}.
\]
}
Iterating this relation gives
\begin{equation}\label{eq:Akrec}
X_{2^k n}=A_9^k\,X_n \qquad (k\ge 0).
\end{equation}
A computation gives
{\scriptsize
\begin{align}
A_9^2&=
\begin{pmatrix}
-3 & -1170 & -153 & 4692 & 891 & -162 & 1 & 0\\
27 & 162 & -3231 & -3060 & 1709 & 306 & -9 & 0\\
-9 & 306 & 1709 & -3060 & -3231 & 162 & 27 & 0\\
1 & -162 & 891 & 4692 & -153 & -1170 & -3 & 0\\
0 & -3 & -1170 & -153 & 4692 & 891 & -162 & 1\\
0 & 27 & 162 & -3231 & -3060 & 1709 & 306 & -9\\
0 & -9 & 306 & 1709 & -3060 & -3231 & 162 & 27\\
0 & 1 & -162 & 891 & 4692 & -153 & -1170 & -3
\end{pmatrix},\label{eq:A9sq}\\[4pt]
A_9^3&=
\begin{pmatrix}
-1143 & -35799 & 26541 & 431613 & 256347 & -3429 & -2385 & -1\\
-81 & 29583 & 163179 & -98277 & -315315 & -58995 & 1369 & 9\\
387 & -6669 & -184785 & -297873 & 109449 & 97273 & 3717 & -27\\
-171 & -9243 & 73737 & 436041 & 232703 & -51057 & -10269 & 3\\
-3 & 10269 & 51057 & -232703 & -436041 & -73737 & 9243 & 171\\
27 & -3717 & -97273 & -109449 & 297873 & 184785 & 6669 & -387\\
-9 & -1369 & 58995 & 315315 & 98277 & -163179 & -29583 & 81\\
1 & 2385 & 3429 & -256347 & -431613 & -26541 & 35799 & 1143
\end{pmatrix},\label{eq:A9cube}
\end{align}
}%
and for $k=4,5,6,7$:
\begin{equation}\label{eq:A9pow}
\begin{aligned}
A_9^4&=2^3\,P_4,\quad &A_9^5&=2^5\,P_5,\\
A_9^6&=2^8\,P_6,\quad &A_9^7&=2^{10}\,P_7,
\end{aligned}
\end{equation}
where $P_4,P_5,P_6,P_7$ are integer matrices whose entries have no common factor of $2$,
and both $P_4$ and $P_6$ have all entries odd.
{\tiny
\setlength{\arraycolsep}{2pt}
\begin{align*}
P_4&=\frac{A_9^4}{2^3}=
\begin{pmatrix}
-3189 & -125001 & 828927 & 3304747 & 305937 & -669771 & -25179 & 297\\
3789 & -61887 & -1341063 & -763155 & 2289783 & 647091 & -33021 & -161\\
-1269 & 136575 & 377711 & -2568573 & -2534607 & 157869 & 78741 & -495\\
-963 & -68247 & 776889 & 2948421 & -125145 & -940149 & -40285 & 1287\\
1287 & -40285 & -940149 & -125145 & 2948421 & 776889 & -68247 & -963\\
-495 & 78741 & 157869 & -2534607 & -2568573 & 377711 & 136575 & -1269\\
-161 & -33021 & 647091 & 2289783 & -763155 & -1341063 & -61887 & 3789\\
297 & -25179 & -669771 & 305937 & 3304747 & 828927 & -125001 & -3189
\end{pmatrix},\\[4pt]
P_6&=\frac{A_9^6}{2^8}=
\begin{pmatrix}
-235033 & -9906597 & 134104059 & 481604751 & 5094981 & -142523199 & -7667175 & 176661\\
300753 & -9071683 & -167398947 & -55236951 & 412850019 & 122152455 & -6607953 & -138285\\
-71145 & 17701731 & 29095931 & -425611449 & -412932555 & 27048249 & 15278841 & -57123\\
-183855 & -8577819 & 129905613 & 457815521 & -20616957 & -154149345 & -8209377 & 226395\\
226395 & -8209377 & -154149345 & -20616957 & 457815521 & 129905613 & -8577819 & -183855\\
-57123 & 15278841 & 27048249 & -412932555 & -425611449 & 29095931 & 17701731 & -71145\\
-138285 & -6607953 & 122152455 & 412850019 & -55236951 & -167398947 & -9071683 & 300753\\
176661 & -7667175 & -142523199 & 5094981 & 481604751 & 134104059 & -9906597 & -235033
\end{pmatrix}.
\end{align*}
}

Furthermore, direct matrix powers give
\[
\begin{aligned}
\min_{i,j}\nu_2\bigl((A_9^4)_{ij}\bigr)&=3,&
\min_{i,j}\nu_2\bigl((A_9^5)_{ij}\bigr)&=5,\\
\min_{i,j}\nu_2\bigl((A_9^6)_{ij}\bigr)&=8,&
\min_{i,j}\nu_2\bigl((A_9^7)_{ij}\bigr)&=10,\\
\min_{i,j}\nu_2\bigl((A_9^8)_{ij}\bigr)&=13.
\end{aligned}
\]
Direct computation shows that the minimal polynomial of $A_9$ is
\begin{equation}\label{eq:A9min}
\mu_{A_9}(X)=X^8-6560X^6+8472576X^4-2235564032X^2+68719476736,
\end{equation}
so in particular $\mu_{A_9}(A_9)=0$.  Also,
\[
\begin{aligned}
\nu_2(6560)&=5,\qquad \nu_2(8472576)=11,\\
\nu_2(2235564032)&=22,\qquad \nu_2(68719476736)=36.
\end{aligned}
\]

From $\mu_{A_9}(A_9)=0$ and~\eqref{eq:Akrec}, the $i$-th entry of $X_{2^k n}$ is
$t_9(2^{k+1}n-i)$, so equating the $i$-th entries on both sides yields the scalar recurrence
\[
\begin{aligned}
t_9(2^9 n-i)=\;&6560\,t_9(2^7 n-i)-8472576\,t_9(2^5 n-i)\\
&+2235564032\,t_9(2^3 n-i)-68719476736\,t_9(2n-i),
\end{aligned}
\]
valid for all $n\ge 1$ and $i\in\{1,\ldots,8\}$.
The first eight block vectors are:
{\scriptsize
\[
\begin{aligned}
U_1&=(1,-9,27,-3,-171,387,-81,-1143)^T,\\
U_2&=(2376,-1288,-3960,10296,-7704,-10152,30312,-25512)^T,\\
U_3&=(-18009,71073,-63091,-41445,152307,-114315,-105255,317871)^T,\\
U_4&=(-222816,-252576,665568,-414944,-476640,1183968,-767904,-770400)^T,\\
U_5&=(1960767,-1141047,-1336923,3006531,-1643797,-2245059,4900113,-2608713)^T,\\
U_6&=(-3481128,7519464,-4513320,-4342680,10920888,-7085816,-5454216,14942664)^T,\\
U_7&=(-10273287,-7470657,21907539,-15158523,-10748115,31962987,-24134777,-12645135)^T,\\
U_8&=(45225216,-35400960,-14623488,57957120,-47066880,-18213120,76992768,-60168448)^T.
\end{aligned}
\]
}

\begin{lemma}[The mod $8$ block recurrences]\label{lem:t9mod8}
For every $m\ge 0$ and every $r\in\{0,1,\ldots,7\}$,
\[
t_9(8m+r)\equiv \sum_{k=0}^{7} c_{r,k}\,t_9(m-k)\pmod 8,
\]
where the coefficient vectors $(c_{r,0},\ldots,c_{r,7})$ are, modulo $8$,
\[
(1,1,5,5,3,3,7,7),\ (7,7,3,3,5,5,1,1),\ (3,3,7,7,1,1,5,5),\ (5,5,1,1,7,7,3,3),
\]
\[
(5,5,1,1,7,7,3,3),\ (3,3,7,7,1,1,5,5),\ (7,7,3,3,5,5,1,1),\ (1,1,5,5,3,3,7,7),
\]
in the order $r=0,1,\ldots,7$.
\end{lemma}

\begin{proof}
Applying Lemma~\ref{lem:rec} three times gives the exact recurrences for
$t_9(8m+r)$, equivalently the rows of $A_9^3$ in~\eqref{eq:A9cube}.  Reducing those
rows modulo $8$ gives the coefficient vectors displayed above.
\end{proof}

\begin{lemma}[Residue classes of $X_{n_1}$ modulo $8$]\label{lem:Xmod8}
If $n_1$ is odd, then $X_{n_1}\pmod 8$ depends only on $n_1\pmod 8$ and is given by
\[
X_{n_1}\equiv
\begin{cases}
(7,1,0,0,0,0,0,0)^T,& n_1\equiv 1\pmod 8,\\[2pt]
(3,5,5,3,7,1,0,0)^T,& n_1\equiv 3\pmod 8,\\[2pt]
(0,0,1,7,3,5,5,3)^T,& n_1\equiv 5\pmod 8,\\[2pt]
(0,0,0,0,0,0,1,7)^T,& n_1\equiv 7\pmod 8.
\end{cases}
\pmod 8.
\]
\end{lemma}

\begin{proof}
By Lemma~\ref{lem:t9mod8}, the residue of the block
\[
B_m:=\bigl(t_9(8m),t_9(8m+1),\ldots,t_9(8m+7)\bigr)\pmod 8
\]
depends only on $m\pmod 4$.  Computing the initial values $t_9(0),\ldots,t_9(31)$ gives
\[
B_m\equiv
\begin{cases}
(1,7,3,5,5,3,7,1),& m\equiv 0\pmod 4,\\
(0,0,0,0,0,0,0,0),& m\equiv 1\text{ or }3\pmod 4,\\
(7,1,5,3,3,5,1,7),& m\equiv 2\pmod 4.
\end{cases}
\]
All congruences in this display are taken modulo $8$.

Now write $n_1=2u+1$.  Then
\[
X_{n_1}=\bigl(
\begin{array}{c}
t_9(4u+1),t_9(4u),t_9(4u-1),t_9(4u-2),\\
t_9(4u-3),t_9(4u-4),t_9(4u-5),t_9(4u-6)
\end{array}
\bigr)^T.
\]
These entries are obtained by concatenating the last six entries of
$B_{\lfloor u/2\rfloor-1}$ with the first two entries of $B_{\lfloor u/2\rfloor}$.
We now distinguish the four possibilities for $n_1\pmod 8$.

\smallskip
\noindent If $n_1\equiv 1\pmod 8$, then $u\equiv 0\pmod 4$.  Hence $\lfloor u/2\rfloor$ is even,
so
\[
B_{\lfloor u/2\rfloor}\equiv(1,7,3,5,5,3,7,1),
\qquad
B_{\lfloor u/2\rfloor-1}\equiv 0.
\]
Therefore
\[
X_{n_1}\equiv (7,1,0,0,0,0,0,0)^T\pmod 8.
\]

\smallskip
\noindent If $n_1\equiv 3\pmod 8$, then $u\equiv 1\pmod 4$.  Again $\lfloor u/2\rfloor$ is even,
so
\[
B_{\lfloor u/2\rfloor}\equiv(1,7,3,5,5,3,7,1),
\qquad
B_{\lfloor u/2\rfloor-1}\equiv 0.
\]
Hence
\[
X_{n_1}\equiv (3,5,5,3,7,1,0,0)^T\pmod 8.
\]

\smallskip
\noindent If $n_1\equiv 5\pmod 8$, then $u\equiv 2\pmod 4$.  Thus $\lfloor u/2\rfloor$ is odd,
so
\[
B_{\lfloor u/2\rfloor}\equiv 0,
\qquad
B_{\lfloor u/2\rfloor-1}\equiv(1,7,3,5,5,3,7,1).
\]
Therefore
\[
X_{n_1}\equiv (0,0,1,7,3,5,5,3)^T\pmod 8.
\]

\smallskip
\noindent If $n_1\equiv 7\pmod 8$, then $u\equiv 3\pmod 4$.  Again $\lfloor u/2\rfloor$ is odd,
so
\[
B_{\lfloor u/2\rfloor}\equiv 0,
\qquad
B_{\lfloor u/2\rfloor-1}\equiv(1,7,3,5,5,3,7,1).
\]
Thus
\[
X_{n_1}\equiv (0,0,0,0,0,0,1,7)^T\pmod 8.
\]
This proves the four claimed residue patterns.
\end{proof}

\begin{lemma}[Congruence images of the four residue classes]\label{lem:Ximages}
Let
\[
\xi_1:=(7,1,0,0,0,0,0,0)^T,\quad
\xi_2:=(3,5,5,3,7,1,0,0)^T,
\]
\[
\xi_3:=(0,0,1,7,3,5,5,3)^T,\quad
\xi_4:=(0,0,0,0,0,0,1,7)^T.
\]
Then, for each $s\in\{1,2,3,4\}$,
\[
A_9^2\xi_s\equiv (1,1,1,1,1,1,1,1)^T\pmod 2,
\]
\[
A_9^3\xi_s\equiv (8,8,8,8,8,8,8,8)^T\pmod{16},
\]
\[
P_4\xi_s\equiv (4,4,4,4,4,4,4,4)^T\pmod 8,
\qquad
P_5\xi_s\equiv (8,8,8,8,8,8,8,8)^T\pmod{16},
\]
\[
P_6\xi_s\equiv (4,4,4,4,4,4,4,4)^T\pmod 8,
\qquad
P_7\xi_s\equiv (8,8,8,8,8,8,8,8)^T\pmod{16}.
\]
\end{lemma}

\begin{proof}
By Lemma~\ref{lem:Xmod8}, every odd $n_1$ satisfies $X_{n_1}\equiv \xi_s\pmod 8$ for a unique
$s\in\{1,2,3,4\}$.  For the first congruence, reduce~\eqref{eq:A9sq} modulo $2$:
rows $1$--$4$ become $(1,0,1,0,1,0,1,0)$ and rows $5$--$8$ become
$(0,1,0,1,0,1,0,1)$.  Hence
\[
(A_9^2v)_i\equiv v_1+v_3+v_5+v_7\pmod 2\qquad (1\le i\le 4),
\]
and
\[
(A_9^2v)_i\equiv v_2+v_4+v_6+v_8\pmod 2\qquad (5\le i\le 8).
\]
For each $\xi_s$, both sums are odd:
\[
\begin{aligned}
\xi_1:&\quad 7+0+0+0\equiv 1+0+0+0\equiv 1\pmod 2,\\
\xi_2:&\quad 3+5+7+0\equiv 5+3+1+0\equiv 1\pmod 2,\\
\xi_3:&\quad 0+1+3+5\equiv 0+7+5+3\equiv 1\pmod 2,\\
\xi_4:&\quad 0+0+0+1\equiv 0+0+0+7\equiv 1\pmod 2.
\end{aligned}
\]
This proves the congruence for $A_9^2\xi_s$.

For the remaining congruences, reduce the displayed matrices $A_9^3$, $P_4$, $P_5$, $P_6$ and $P_7$
modulo $16$, $8$, $16$, $8$ and $16$, respectively, and multiply by each of the four vectors
$\xi_1,\xi_2,\xi_3,\xi_4$.  These computations give exactly the congruences stated in the lemma.
\end{proof}

\begin{proof}[Proof of Theorem~\ref{thm:m9}]
Let
\[
g_9(r):=5\Bigl\lceil\frac{r}{2}\Bigr\rceil-2(r\bmod 2).
\]
We prove by induction on $n\ge 0$ that
\[
\nu_2\bigl(t_9(8n+j)\bigr)=g_9\bigl(\nu_2(n+1)\bigr),
\qquad j\in\{0,1,\ldots,7\}.
\]
The initial cases $0\le n\le 7$ follow from the explicit block vectors
$U_1,\ldots,U_8$ listed above.  Writing these valuations out,
\[
\begin{aligned}
&\nu_2(U_1)=\nu_2(U_3)=\nu_2(U_5)=\nu_2(U_7)=(0,0,0,0,0,0,0,0),\\
&\nu_2(U_2)=\nu_2(U_6)=(3,3,3,3,3,3,3,3),\\
&\nu_2(U_4)=(5,5,5,5,5,5,5,5),\qquad
\nu_2(U_8)=(8,8,8,8,8,8,8,8).
\end{aligned}
\]
Since
\[
g_9\bigl(\nu_2(1)\bigr)=0,\ g_9\bigl(\nu_2(2)\bigr)=3,\ g_9\bigl(\nu_2(3)\bigr)=0,\ g_9\bigl(\nu_2(4)\bigr)=5,
\]
\[
g_9\bigl(\nu_2(5)\bigr)=0,\ g_9\bigl(\nu_2(6)\bigr)=3,\ g_9\bigl(\nu_2(7)\bigr)=0,\ g_9\bigl(\nu_2(8)\bigr)=8,
\]
the claimed formula holds for $n=0,1,\ldots,7$.

Write $n+1=2^r n_1$ with $n_1$ odd.  For each $j\in\{0,\ldots,7\}$ we have
\[
t_9(8n+j)=t_9\bigl(2^{r+3}n_1-8+j\bigr)=t_9\bigl(2\cdot 2^{r+2}n_1-(8-j)\bigr),
\]
so $t_9(8n+j)$ is the $(8-j)$-th entry of $X_{2^{r+2}n_1}$.
Hence it is enough to prove that
\[
\nu_2\bigl((X_{2^{r+2}n_1})_i\bigr)=g_9(r)\quad\text{for all }i\in\{1,\ldots,8\}.
\]

\noindent\textbf{Case $r=0$.}
Then $n$ is even and $g_9(0)=0$, so we must show that all entries of $X_{4n_1}$ (with $n_1:=n+1$ odd) are odd.

By Lemma~\ref{lem:Xmod8}, the vector $X_{n_1}\pmod 8$ is one of the four vectors
$\xi_1,\xi_2,\xi_3,\xi_4$ from Lemma~\ref{lem:Ximages}.  Therefore
Lemma~\ref{lem:Ximages} gives
\[
X_{4n_1}=A_9^2X_{n_1}\equiv (1,1,1,1,1,1,1,1)^T\pmod 2.
\]
Hence every entry of $X_{4n_1}$ is odd, as required.

\medskip

\noindent\textbf{Case $r=1$.}
Then $n+1=2n_1$ with $n_1$ odd.  We show that every entry of $X_{8n_1}$ has valuation $g_9(1)=3$.

By~\eqref{eq:Akrec},
\[
X_{8n_1}=A_9^3X_{n_1}.
\]
By Lemmas~\ref{lem:Xmod8} and~\ref{lem:Ximages}, in every case
\[
A_9^3X_{n_1}\equiv (8,8,8,8,8,8,8,8)^T\pmod{16}.
\]
Hence every entry of $X_{8n_1}$ is divisible by $8$ but not by $16$, and therefore
\[
\nu_2\bigl((X_{8n_1})_i\bigr)=3=g_9(1)
\qquad (i=1,\ldots,8).
\]

\medskip

\noindent\textbf{Case $r=2$.}
Then $n+1=4n_1$ with $n_1$ odd.  Our goal is to prove that each entry of $X_{16n_1}$ has
valuation $g_9(2)=5$.

By~\eqref{eq:Akrec} and~\eqref{eq:A9pow},
\[
X_{16n_1}=A_9^4X_{n_1}=2^3P_4X_{n_1}.
\]
By Lemmas~\ref{lem:Xmod8} and~\ref{lem:Ximages}, in every case
\[
P_4X_{n_1}\equiv (4,4,4,4,4,4,4,4)^T\pmod 8.
\]
Hence each entry of $P_4X_{n_1}$ is divisible by $4$ but not by $8$, and therefore every
entry of $X_{16n_1}=2^3P_4X_{n_1}$ is divisible by $2^5$ but not by $2^6$. Thus
\[
\nu_2\bigl((X_{16n_1})_i\bigr)=5=g_9(2)
\qquad (i=1,\ldots,8).
\]

\medskip

\noindent\textbf{Case $r=3$.}
Then $n+1=8n_1$ with $n_1$ odd.  We claim that all entries of $X_{32n_1}$ have valuation
$g_9(3)=8$.

By~\eqref{eq:Akrec} and~\eqref{eq:A9pow},
\[
X_{32n_1}=A_9^5X_{n_1}=2^5P_5X_{n_1}.
\]
By Lemmas~\ref{lem:Xmod8} and~\ref{lem:Ximages}, in every case
\[
P_5X_{n_1}\equiv (8,8,8,8,8,8,8,8)^T\pmod{16}.
\]
Hence each entry of $P_5X_{n_1}$ is divisible by $8$ but not by $16$, and therefore every
entry of $X_{32n_1}=2^5P_5X_{n_1}$ is divisible by $2^8$ but not by $2^9$. Thus
\[
\nu_2\bigl((X_{32n_1})_i\bigr)=8=g_9(3)
\qquad (i=1,\ldots,8).
\]

\medskip

\noindent\textbf{Case $r=4$.}
Then $n+1=16n_1$ with $n_1$ odd.  We show that every entry of $X_{64n_1}$ has
valuation $g_9(4)=10$.

By~\eqref{eq:Akrec} and~\eqref{eq:A9pow},
\[
X_{64n_1}=A_9^6X_{n_1}=2^8P_6X_{n_1}.
\]
By Lemmas~\ref{lem:Xmod8} and~\ref{lem:Ximages}, in every case
\[
P_6X_{n_1}\equiv (4,4,4,4,4,4,4,4)^T\pmod 8.
\]
Hence each entry of $P_6X_{n_1}$ is divisible by $4$ but not by $8$, and therefore every
entry of $X_{64n_1}=2^8P_6X_{n_1}$ is divisible by $2^{10}$ but not by $2^{11}$. Thus
\[
\nu_2\bigl((X_{64n_1})_i\bigr)=10=g_9(4)
\qquad (i=1,\ldots,8).
\]

\medskip

\noindent\textbf{Case $r=5$.}
Then $n+1=32n_1$ with $n_1$ odd.  Our goal is to prove that all entries of $X_{128n_1}$
have valuation $g_9(5)=13$.

By~\eqref{eq:Akrec} and~\eqref{eq:A9pow},
\[
X_{128n_1}=A_9^7X_{n_1}=2^{10}P_7X_{n_1}.
\]
By Lemmas~\ref{lem:Xmod8} and~\ref{lem:Ximages}, in every case
\[
P_7X_{n_1}\equiv (8,8,8,8,8,8,8,8)^T\pmod{16}.
\]
Hence each entry of $P_7X_{n_1}$ is divisible by $8$ but not by $16$, and therefore every
entry of $X_{128n_1}=2^{10}P_7X_{n_1}$ is divisible by $2^{13}$ but not by $2^{14}$. Thus
\[
\nu_2\bigl((X_{128n_1})_i\bigr)=13=g_9(5)
\qquad (i=1,\ldots,8).
\]

\medskip

\noindent\textbf{Case $r=6$.}
Then $n+1=64n_1$ with $n_1$ odd.  We show that each entry of $X_{256n_1}$ has valuation
$g_9(6)=15$.

Applying~\eqref{eq:Akrec} eight times from $X_{n_1}$ and using the minimal
polynomial~\eqref{eq:A9min}, we obtain
\[
X_{256n_1}=6560\,X_{64n_1}-8472576\,X_{16n_1}+2235564032\,X_{4n_1}-68719476736\,X_{n_1}.
\]
The four vectors on the right correspond to the already established cases $r=4,2,0,0$,
so their entries have valuations $g_9(4)=10$, $g_9(2)=5$, $g_9(0)=0$, and $g_9(0)=0$, respectively.
Therefore the four terms have valuations
\[
5+g_9(4)=15=g_9(6),\qquad 11+g_9(2)=16,\qquad 22+g_9(0)=22,\qquad 36+g_9(0)=36.
\]
Hence the first term has strictly the smallest $2$-adic valuation, and so
\[
\nu_2\bigl((X_{256n_1})_i\bigr)=15=g_9(6)
\qquad (i=1,\ldots,8).
\]

\medskip
\noindent\textbf{Case $r\ge 7$.}
Apply~\eqref{eq:Akrec} eight times starting from $X_{2^{r-6}n_1}$:
\[
X_{2^{r+2}n_1}=A_9^8\,X_{2^{r-6}n_1}.
\]
Using the minimal polynomial~\eqref{eq:A9min} this becomes
\begin{equation}\label{eq:X8step}
\begin{aligned}
X_{2^{r+2}n_1}
=\;&6560\,X_{2^r n_1}-8472576\,X_{2^{r-2}n_1}\\
&+2235564032\,X_{2^{r-4}n_1}-68719476736\,X_{2^{r-6}n_1}.
\end{aligned}
\end{equation}
The first three vectors correspond (by the same translation) to blocks
$2^{r-2}n_1-1$, $2^{r-4}n_1-1$, and $2^{r-6}n_1-1$, all strictly less than $n$.
Hence the induction hypothesis gives
\[
\nu_2\bigl((X_{2^r n_1})_i\bigr)=g_9(r-2),\qquad
\nu_2\bigl((X_{2^{r-2} n_1})_i\bigr)=g_9(r-4),
\]
\[
\nu_2\bigl((X_{2^{r-4} n_1})_i\bigr)=g_9(r-6).
\]
We now separate the last term.

If $r=7$, we only use the crude bound
\[
\nu_2\bigl((X_{2n_1})_i\bigr)\ge 0.
\]
If $r\ge 8$, then the last vector also corresponds to the smaller block
$2^{r-8}n_1-1$, and the induction hypothesis gives
\[
\nu_2\bigl((X_{2^{r-6} n_1})_i\bigr)=g_9(r-8).
\]
Since $g_9(r)-g_9(r-2)=5$ for all $r$, the four terms in~\eqref{eq:X8step} have
$2$-adic valuations
\begin{align*}
\nu_2(6560)+g_9(r-2)&=5+(g_9(r)-5)=g_9(r),\\
\nu_2(8472576)+g_9(r-4)&=11+(g_9(r)-10)=g_9(r)+1,\\
\nu_2(2235564032)+g_9(r-6)&=22+(g_9(r)-15)=g_9(r)+7,\\
\nu_2(68719476736)+g_9(r-8)&=36+(g_9(r)-20)=g_9(r)+16\qquad (r\ge 8),
\end{align*}
while for $r=7$ the last term has valuation at least $36>g_9(7)=18$.
The first term has strictly the smallest $2$-adic valuation, so
\[
\nu_2\bigl((X_{2^{r+2}n_1})_i\bigr)=g_9(r)=g_9\bigl(\nu_2(n+1)\bigr),
\]
completing the induction.
\end{proof}

\section*{Acknowledgements}
The author thanks Liu Sihan for helpful comments and suggestions.
This work was supported by the National Natural Science Foundation of China
(No.~12301011).

\end{document}